\documentclass[reqno,10pt]{amsart}
\usepackage{amscd,amssymb,amsmath,color,stmaryrd,defs}
\usepackage{latexsym}
\usepackage[all]{xy}
\usepackage[colorlinks=true]{hyperref}

\def\oFF{{\ol\FF}}
\def\Fr{{\rm Fr}}
\def\rholog{\rho^\rmlog}

\def\AA{{\bbA}}
\def\NN{{\bbN}}
\def\LL{{\bbL}}
\def\ZZ{{\bbZ}}
\def\CC{{\bbC}}
\def\QQ{{\bbQ}}
\def\RR{{\bbR}}
\def\FF{{\bbF}}
\def\.{,\dots,}

\def\whLa{{\wh{L^a}}}

\def\spl{{\rm spl}}

\def\topdeg{{\rm tr.c}}
\def\Topdeg{{\rm Tr.c}}
\def\rmlog{{\rm log}}

\def\whKa{{\wh{K^a}}}

\begin{document}

\author{Michael Temkin}
\title{Small extensions of analytic fields}
\thanks{This research was supported by ERC Consolidator Grant 770922 - BirNonArchGeom, by the Israel Science Foundation Grant 1203/22, and by a grant from the School of Mathematics at the Institute for Advanced Study.}
\address{Einstein Institute of Mathematics, The Hebrew University of Jerusalem, Giv'at Ram, Jerusalem, 91904, Israel}
\email{temkin@math.huji.ac.il}
\begin{abstract}
We study basic properties of transcendental extensions of analytic (i.e. real valued complete) fields and the related notions of transcendence degree. Such an extension $K/k$ is called small if it is topologically-algebraically generated by finitely many elements. We prove that this property is inherited by subextensions and hence topological generating degree of small extensions is monotonic. Much more detailed results are obtained in the case of degree one. Let $k$ be an analytic algebraically closed field of positive residual characteristic $p$ and $K=\wh{k(t)^a}$ with a non-trivial valuation. In a previous work it was shown that the set $I_{K/k}$ of intermediate complete algebraically closed subextensions $k\subseteq F\subseteq K$ is totally ordered by inclusion. In this paper we show that $I_{K/k}$ is an interval parameterized by the distance between $t$ and $F$. Moreover, logarithmic parameterizations induced by other generators differ by PL functions with slopes in $p^\ZZ$ and corners in $|K^\times|$, so $I_{K/k}$ acquires a natural PL structure.
\end{abstract}

\maketitle
%\tableofcontents

\section{Introduction}

\subsection{Recollections}\label{recsec}
Let us recall main results about topological degree that were proved in \cite{topdeg}. These results will be used without further referencing throughout the paper. We will use the abbreviation {\em t.a.} instead of topologically algebraically or topgebraically, as in \cite{topdeg}. Given an extension $L/K$ of analytic fields, a set $S\subset L$ is called {\em t.a. independent} over $K$ if $\wh{K(S')^a}\subsetneq\wh{K(S)^a}$ for any $S'\subsetneq S$, and it is called {\em t.a. generating} if $\whLa=\wh{k(S)^a}$. The minimum of cardinals of t.a. generating sets of $L/K$ is called the {\em topological generating degree} $\Topdeg(L/K)$ and the supremum of all cardinals of t.a. independent sets is called the {\em topological transcendence degree} $\topdeg(L/K)$. For shortness, we use notation similar to Huber's topological transcendence degree $\topdeg$, see \cite[\S1.8]{Huberbook} and \cite[\S21]{diamonds}, but we warn the reader that our terminology is different -- Huber's transcendence degree is the generating degree in our meaning which is denoted $\Topdeg$. Since we consider the two cardinals, our terminology seems more natural. By {\em small} extensions we mean t.a. finitely generated extensions, that is, extensions $L/K$ with $\Topdeg(L/K)<\infty$. Naturally, the extensions with $\Topdeg(L/K)=\infty$ will be called {\em large}.

It is almost obvious that for a tower $F/L/K$ one invariant is monotonic and another one is subadditive: $\topdeg(L/K)\le\topdeg(F/K)$ and $\Topdeg(F/K)\le\Topdeg(L/K)+\Topdeg(F/L)$, but any other property is not clear at all. In fact, one of the main goals of this paper and my motivation to continue the research of \cite{topdeg} was proving Theorem~\ref{monotonic}, which states that $\Topdeg$ is monotonic for small extensions, and essentially reduces to proving that being a small extension is preserved by passing to subextensions. This answers a natural question in Berkovich and adic geometries which was (somewhat implicitly) open since their foundation, and was explicitly stated by Scholze in \cite[Question~21.4]{diamonds}. I am very grateful to Peter Scholze for an email exchange and pointing out that the results of \cite{topdeg} were not sufficient to completely answer this question.

It is shown in \cite{topdeg} that a version of exchange lemma holds for t.a. dependency, hence $\topdeg(L/K)\le\Topdeg(L/K)$, see \cite[Theorem~3.2.1]{topdeg} and its proof. Of course, the two cardinals are equal if there exists a topological transcendence basis -- a set, which is both t.a. independent and generating, and it always exists for small extensions by \cite[Theorem~3.2.3]{topdeg}, but large extensions, such as spherical completions, tend not to have such a basis, see \cite[Theorem~5.3.2]{topdeg}. If the residual characteristic $p=\cha(\tilK)$ is positive, the situation is even stranger as there exist the following pathologies, which are essentially equivalent: there exist non-invertible $\whKa$-endomorphisms of $\wh{K(x)^a}$, there exist towers $\whKa\subsetneq L=\wh{K(x)^a}\subsetneq F=\wh{K(y)^a}$, there exist large extensions $L/K$ which satisfy $\topdeg(L/K)=1$. In particular, the topological transcendence degree is not additive in towers even for small extensions. Finally, it was observed in \cite[Theorem~5.3.9]{topdeg} that if $L$ and $K$ are algebraically closed and $\topdeg(L/K)=1$, then the set $I_{L/K}=\{F_i\}$ of analytic algebraically closed intermediate fields $K\subseteq F_i\subseteq L$ is totally ordered.

\subsection{Main results}
In this paper we answer questions about extensions of finite topological transcendence degree, which were left open or unexplored in \cite{topdeg}. In particular, we prove that if $L/K$ is a small extension of analytic algebraically closed fields, then any intermediate algebraically closed analytic field $F$ is small over $K$, and hence satisfies $\Topdeg(F/K)\le\Topdeg(L/K)$, see Theorem~\ref{monotonic}. Moreover, the equality holds if and only if $F$ and $L$ are $K$-isomorphic, see Theorem~\ref{similarth}. In fact, the claim about small extensions follows by applying Baire category theorem, once we prove the following result, which is of its own interest: any chain of algebraically closed analytic intermediate fields has at most countable cofinality, see Theorem~\ref{subfieldth}. The latter result is reduced to the case of topological transcendence degree one, already covered by Lemma~\ref{subfieldcor}, but later on we prove in Theorem~\ref{Stheorem} the following complete and simple description of $I_{L/K}$ which was unanticipated in \cite{topdeg}: any t.a. generator $t\in L$ induces a bijection $\rho_t\:I_{L/K}\toisom[0,r_K(t)]$ by sending an intermediate field $F$ to $r_F(t)$. In fact, this is a really simple result, which was overlooked in \cite{topdeg} and noticing it recently made all progress of the current paper possible.

Furthermore, we study the following natural question in \S\ref{subfieldsec}: what is the natural structure of the interval $I_{L/K}$ for $L=\wh{K(t)^a}$? We identify it with $[-\log(r),\infty)$ by $-\log\circ\rho_t$ and show in Theorem~\ref{PLth} that any other t.a. generator $t'$ induces a parametrization which differs from this one by a PL-function with slopes in $p^\ZZ$ and corners in $|L^\times|$ (satisfying an addition restriction at the endpoint when $|K^\times|\neq|L^\times|$). Moreover, in Theorem~\ref{PLth2} we show that the converse is also true, obtaining a complete characterization of the set of possible parameterizations of $I_{L/K}$. In particular, $I_{L/K}$ acquires a natural $\ZZ[\frac{1}{p}]$-structure.

Finally, a good control on the subfields of small extensions enables us to study in \S\ref{largesec} large extensions $L/K$ of finite topological transcendence degree and the set of their algebraically closed subfields. In particular, in this case $\Topdeg(L/K)$ is necessarily the first uncountable cardinal $\aleph_1$ and if $d=1$, then $I_{L/K}$ is nothing else but an Alexandroff's long ray.

\section{The monotonicity theorem}
%In this section we will answer various questions which were left open in \cite{topdeg}. This mainly concerns the properties of the topological generating degree $\Topdeg$. In particular, we will prove that it is monotonic in the finite case and that the only case when $\topdeg(K/k)=d<\kappa=\Topdeg(K/k)$ with a finite $d$ happens when $\kappa=\aleph_1$.
In this section we present the fastest route (we know) to prove the monotonicity theorem for small extensions.

\subsection{Countably generated extensions}

\subsubsection{Completed filtered unions}
We will use completed filtered unions of analytic fields, and the question whether the completion is non-trivial will be essential.

\begin{lem}\label{cofinalem}
Assume $K$ is an analytic field and $\{F_i\}_{i\in I}$ is a filtered family of analytic subfields, then:

(i) If $I$ is of countable cofinality and not all $F_i$ are trivially valued, then $F=\cup_{i\in I}F_i$ is not complete.

(ii) Conversely, $F$ is complete in any of the following cases: the cofinality is one, the fields are trivially valued, or any countable subset in $I$ is bounded (for example, the family is a chain of an uncountable cofinality).
\end{lem}
\begin{proof}
(i) Choose a cofinal increasing subsequence, which we denote for shortness $F_0\subsetneq F_1\subsetneq\dots$. Then $F=\cup_{j=0}^\infty F_j$ and each $F_j$ is nowhere dense in $\hatF$. Therefore $F\neq\hatF$ by Baire category theorem.

(ii) The first two cases are trivial, so assume that any countable subset is bounded. Any Cauchy sequence $a_0,a_1,\dots$ of elements of $F$ lies already in a countable union of $F_i$'s, and hence is contained in a single field $F_j$. So, the limit already lies in some $F_j$.
\end{proof}

\begin{rem}
If $F_0\subsetneq F_1\subsetneq\dots$ is a strictly increasing sequence of analytic $k$-fields such that $F_i=\wh{k(x_i)^a}$ and $F$ is the completed union, then $\topdeg(F/k)=1$ because any pair $\{x,y\}\subset F$ which is t.a. independent over $k$ can be slightly deformed by \cite[Theorem~3.1.9]{topdeg} but no $F_i$ contains a t.a. independent pair since $\topdeg(F_i/k)=1$. It follows that any $x\in F\setminus\cup_i F_i$ is a t.a. generator of $F$ over $k$ because for any $i$ the pair $x,x_i$ is t.a. dependent, $x\notin F_i$, and hence $x_i\in\wh{k(x)^a}$. Already this simple observation was missed in \cite{topdeg}. In particular, it implies that the answer to \cite[Question~5.3.7]{topdeg} is negative -- for any $k$-endomorphism $\phi$ of $K=\wh{k(t)^a}$ the completed union of $\phi^{-n}(K)$ is again generated by a single element.
\end{rem}

The argument from the above remark extends to countable unions of analytic $k$-fields of bounded topological generating degree.

\begin{theor}\label{unionth}
Let $K/k$ be an extension of analytic fields, $\{F_i\}_{i\in I}$ a filtered family of intermediate fields of at most countable cofinality and $F$ the completion of $\cup_i F_i$. If $d=\max_i(\Topdeg(F_i/k))$ is finite, then $\Topdeg(F/k)=d$.
\end{theor}
\begin{proof}
We can replace $K$ and all other fields by $\whKa$ and completed algebraic closures in $\whKa$ as this does not change the invariant $\Topdeg$. So in the sequel all fields are assumed to be algebraically closed. Note that $\topdeg(F/k)=d$ and hence $\topdeg(L/k)\le d$ for any intermediate extension $k\subseteq L\subseteq F$. Indeed, otherwise there exist elements $x_0\.x_d\in F$ t.a. independent over $k$, and by \cite[Theorem~3.1.9]{topdeg} a small enough perturbation $x'_0\. x'_d$ of $x_0\.x_d$ is also t.a. independent over $k$. In particular, we can achieve that $x'_0\.x'_d$ lie in $\cup_iF_i$ and hence also in some $F_i$, which contradicts that $\topdeg(F_i/k)=\Topdeg(F_i/k)\le d$.

We will run induction on $d$ with the base $d=0$ being obvious. Since $\topdeg(F_i/k)=\Topdeg(F_i/k)$, we have that $\Topdeg(F_i/K)\le \Topdeg(F_j/K)$ for $i\le j$. So, we can replace the family by a cofinal sequence $F_0\subseteq F_1\subseteq\dots$ such that $\Topdeg(F_i/k)=d$ for any $i$. If the sequence stabilizes, the claim is obvious, so assume that it does not and hence by Lemma~\ref{cofinalem} there exists $x\in F\setminus\cup_iF_i$. Set $k'=\wh{k(x)^a}$ and $F'_i=\wh{F_i(x)^a}$ and note that $F$ is the completed union of $F'_i$. We claim that $\Topdeg(F'_i/k')\le d-1$, and once we prove this the theorem will follow as follows: by the induction assumption $\Topdeg(F/k')\le d-1$ and hence $\Topdeg(F/k)\le d-1+\Topdeg(k'/k)\le d$. Therefore $\topdeg(F/k)=\Topdeg(F/k)$ and the monotonicity of the invariant $\topdeg$ implies that in fact $\Topdeg(F/k)=d$.

To prove the claim we choose a topological transcendence basis $x_1\.x_d$ of the small extension $F_i/k$. Then $F'_i$ is t.a. generated over $k$ by the set $S=\{x,x_1\.x_d\}$ and since $\topdeg(F'_i/k)\le d$, there exists a t.a. dependency between the elements. By our construction, $x\notin F_i\supseteq\wh{k(x_1\.x_d)^a}$ and hence there exists $x_m$ lying in $\wh{k(S\setminus\{x_m\})^a}$. Therefore $F'_i=\wh{k(S)^a}=\wh{k(S\setminus\{x_m\})^a}$ and hence $S\setminus\{x,x_m\}$ t.a. generates $F'_i$ over $k'$. In particular, $\Topdeg(F'_i/k')\le d-1$.
\end{proof}

As an immediate corollary we can now extend \cite[Theorem 3.2.3]{topdeg} to the countable cardinal. As we will later see, this is already optimal, as the result fails for the first uncountable cardinal.

\begin{cor}\label{equalcor}
If $K/k$ is an extension of analytic fields such that $\Topdeg(K/k)\le\aleph_0$, then $\topdeg(K/k)=\Topdeg(K/k)$.
\end{cor}
\begin{proof}
Since $\topdeg(K/k)\le\Topdeg(K/k)$, we should only rule out the possibility that simultaneously $\topdeg(K/k)=d$ and $\Topdeg(K/k)=\aleph_0$. In the latter case $\whKa$ is the completion of $k(x_0,x_1,\dots)^a$ and hence $K$ is the completed union of the sequence $K_n=\wh{k(x_0\.x_n)^a}$. Since $\Topdeg(K_n/k)$ is finite, it coincides with $\topdeg(K_n/k)$ and hence is bounded by $d$. Thus $\Topdeg(K/k)\le d$ by Theorem~\ref{unionth} yielding a contradiction.
\end{proof}

\subsection{Cofinality of subfields of small extensions}
Now we will prove that algebraically closed subextensions of small extensions have countable cofinality.

\subsubsection{One generator}
We start with extensions with a single t.a. generator. The following self-similarity result is hidden in the proof of \cite[Theorem~5.2.2]{topdeg}, so we formulate and prove it for the sake of completeness. As in \cite{topdeg} given an extension $L/K$ and an element $x\in L$ by $r_K(x)=\inf_{c\in K^a}|x-c|$ we denote the distance between $x$ and $K^a$ evaluated in $L^a$.

\begin{lem}\label{subfieldlem}
Assume that $L/K/k$ is a tower of algebraically closed analytic fields such that $\Topdeg(K/k)=\Topdeg(L/k)=1$. Then the analytic $k$-fields $K$ and $L$ are isomorphic and the extension $L/K$ is immediate.
\end{lem}
\begin{proof}
We can assume that the valuation on $K$ is non-trivial as the other case is easily ruled out by classical tools. Choose generators $K=\wh{k(x)^a}$ and $L=\wh{k(y)^a}$. Since $k(y)$ is dense in $L$ and $r_k(x)>0$, we can find an element $y'\in k(y)^a$ such that $|x-y'|<r_k(x)$. Then $k(y)^a$ is algebraic over $k(y')$ and hence $L=\wh{k(y')^a}$. In addition, $\wh{k(y')}\toisom\wh{k(x)}$ by \cite[Lemma~3.1.6]{topdeg}, and hence the analytic $k$-fields $K$ and $L$ are also isomorphic. The latter also implies that $E_{L/K}=E_{L/k}-E_{K/k}=0$ and $F_{L/K}=F_{L/k}-F_{K/k}=0$, where we use the notation $F_{L/K}=\trdeg(\tilL/\tilK)$ and $E_{L/K}$ is the $\QQ$-dimension of $|L^\times|^\QQ/|K^\times|^\QQ$. Since $\tilk$ is algebraically closed and $|K^\times|$ is divisible this implies that $\tilL=\tilK$ and $|L^\times|=|K^\times|$, that is, the extension $L/K$ is immediate.
\end{proof}

Now, we can bound the ordered set $I_{K/k}$.

\begin{lem}\label{subfieldcor}
Let $K/k$ be an extension of algebraically closed analytic fields such that $\Topdeg(K/k)=1$ and let $I_{K/k}=\{F_i\}$ denote the set of analytic algebraically closed intermediate fields $k\subseteq F_i\subseteq K$. Then any t.a. generator $t\in K$ induces an injective map $\rho_t\:I_{K/k}\into[0,r_k(t)]$, which sends $F\in I_{K/k}$ to $r_F(t)$.
\end{lem}
\begin{proof}
We should prove that if $l\subsetneq L$ are two different elements of $I_{K/k}$, then the inequality $r_L(t)\le r_l(t)$ is strict. We will argue by a contradiction, so assume that $r_L(t)=r_l(t)$. Since $K=\wh{l(t)^a}$, we can replace $k$ by $l$. Thus, we assume in the sequel that $k=l$ and $r_k(t)=r_L(t)$. In addition, decreasing $L$ can only increase $r_L(t)$, so we can replace $L$ by a subfield of the form $\wh{k(y)^a}$ with $y\in L\setminus k$. Then we still have that $r_L(t)=r_k(t)$ and by Lemma \ref{subfieldlem} there exists a $k$-isomorphism of analytic fields $\phi\:L=\wh{k(t)^a}\toisom K$.

Set $x=\phi(t)$ and note that $|x-t|\le r_k(x)=r_k(t)$. Since $r_k(t)=r_L(t)\le|x-t|$ we obtain that in fact $|x-t|=r_L(t)$. This implies that the valuation on $L(x-t)$ is a generalized Gauss valuation. In particular, $L(t)/L$ is not immediate, hence $K/L$ is not immediate and this contradicts the second assertion of Lemma~\ref{subfieldlem}.
\end{proof}

\begin{rem}
We will prove later by an explicit construction (essentially taken from \cite{topdeg}) that each $\rho_t$ is a bijection. But we even do not need the surjectivity to prove the monotonicity theorem, as already at this stage one can easily show using Theorem~\ref{unionth} that any intermediate field $F_i$ is of the form $\wh{k(x_i)^a}$.
\end{rem}

\subsubsection{Finitely many generators}
Even more generally than in the above remark, we can now bound subextensions of an arbitrary small extension.

\begin{theor}\label{subfieldth}
Let $K/k$ be a small extension of algebraically closed analytic fields. Then any chain $\{F_i\}_{i\in I}$ of intermediate algebraically closed analytic fields is of at most countable cofinality.
\end{theor}
\begin{proof}
Assume to the contrary that the cofinality is uncountable. Decreasing $I$ we can then assume that it is well-ordered and uncountable. We will run induction on $d=\Topdeg(K/k)$, with the case of $d=0$ being trivial. So, assume that $d\ge 1$ and choose a topological transcendence basis $x=x_1\.x_d$ of $K/k$. Set $k'=\wh{k(x)^a}$ and note that $\Topdeg(K/k')=d-1$. The fields $F'_i=\wh{F_i(x)^a}$ form a chain of intermediate subfields of $K/k'$, hence by induction assumption this chain stabilizes at some $i_0$ such that there exist at most countably many elemenets $i<i_0$. Removing these elements and renaming $i_0$ by $0$ we can assume that $F'_i=F'_0$ for any $i\in I$. Therefore, $\{F_i\}_{i\in I}$ is an uncountable well-ordered chain of intermediate fields for the extension $\wh{F_0(x)^a}/F_0$ and by Corollary~\ref{subfieldcor}, there exists an order reversing embedding of $I$ into the interval $[0,r_{F_0}(x)]$, which is an absurd.
\end{proof}

\subsection{Finite monotonicity of $\Topdeg$}

\subsubsection{The monotonicity theorem}
As a corollary we can prove that the topological generating degree is monotonic in the finite case. In particular, this provides a positive answer to \cite[Question~21.4]{diamonds}.

\begin{theor}\label{monotonic}
Let $K\subseteq E\subseteq F\subseteq L$ be a tower of analytic fields such that the extension $L/K$ is small. Then $\Topdeg(F/E)\le\Topdeg(L/K)$, in particular, $F/E$ is small.
\end{theor}
\begin{proof}
We should prove that $\Topdeg(F/K)\le\Topdeg(L/K)$, as $\Topdeg(F/E)\le\Topdeg(F/K)$ in the obvious way. Replacing the fields with their completed algebraic closures does not affect the invariants $\Topdeg$, so we can assume that the fields are algebraically closed. If $\Topdeg(F/K)<\infty$, then the topological generating degrees coincide with the topological transcendence degree, and the assertion follows from monotonicity of $\topdeg$. Thus, we should only rule out the possibility of $\Topdeg(F/K)=\infty$.

By transfinite induction there exists a maximal strictly increasing chain $F_i$ ordered by ordinals $i\le \omega$, where $\omega$ is the maximal ordinal in the chain, such that $F_0=K$, for any $i<\omega$ one has that $F_{i+1}=\wh{F_i(x_i)^a}$, and for any limit ordinal $j\le\omega$ the field $F_j$ is the completion of $\cup_{i<j}F_i$. In particular, $F_\omega=F$ by the maximality assumption. Then the cardinal $\omega$ is countable by Theorem~\ref{subfieldth}, hence $F$ is topologically generated over $k$ by countably many elements $x_i$, and therefore by Corollary~\ref{equalcor} $\Topdeg(F/K)=\topdeg(F/K)\le\topdeg(L/K)<\infty$.
\end{proof}

\subsubsection{The self-similarity theorem}
Also, one can now strengthen the self-similarity lemma by extending it to any finite topological degree.

\begin{theor}\label{similarth}
Let $L/K$ be a small extension of algebraically closed analytic fields and $F$ an intermediate algebraically closed analytic field. Then $L$ and $F$ are isomorphic as analytic $K$-fields if and only if the inequality $\topdeg(F/K)\le\topdeg(L/K)$ is an equality. In addition, in this case the extension $L/F$ is immediate.
\end{theor}
\begin{proof}
Only the inverse implication is non-trivial, so assume that both topological transcendence degrees are equal to $d$.  We claim that for an element $y\in L\setminus F$ the $K$-field $E=\wh{F(y)^a}$ is isomorphic to $F$ and the extension $E/F$ is immediate. Since $L$ is t.a. finitely generated over $F$ (even over $K$), the theorem will follow by induction on the number of t.a. generators.

Now, let us prove the claim. By Theorem~\ref{monotonic} $F/K$ is small, hence possesses a topological transcendence basis $x_1\.x_d$. The set $\{y,x_1\.x_d\}$ is t.a. dependent over $K$, and up to renumbering the elements $x_1\in \wh{K(y,x_2\.x_d)^a}$. In particular, the latter field coincides with $E$, and setting $K'=\wh{K(x_2\.x_d)^a}$ we obtain that $F=\wh{K'(x_1)^a}\subset E=\wh{K'(y)^a}$. So, by Lemma~\ref{subfieldlem} $F$ and $E$ are isomorphic even as analytic $K'$-fields and the extension $E/F$ is immediate.
\end{proof}

\section{Subfields of $\wh{k(t)^a}$}\label{subfieldsec}
Throughout this section $K/k$ is an extension of algebraically closed analytic fields such that $p=\cha(\tilk)>0$, $|K^\times|\neq 1$ and $\Topdeg(K/k)=1$. Often we will also fix a t.a. generator $t$ and then $K=\wh{k(t)^a}$ and $r=r_k(t)>0$. Our goal is to study the set $I_{K/k}$ of intermediate algebraically closed analytic fields.

\subsection{Ordered structure of $I_{K/k}$}
We already know that $\rho_t\:I_{K/k}\into[0,r_k(t)]$ by Lemma~\ref{subfieldcor}, so our goal is prove that it is onto by constructing an intermediate field $F$ of a prescribed radius $r_F(t)$. The trick which was already used in \cite{topdeg} is to construct an extension $L/K$ with $\Topdeg(L/k)=1$ instead, and then use an isomorphism $L\toisom K$ which sends $K$ to a subfield $F$.

\subsubsection{Endomorphisms}
We start with recalling how endomorphisms of $\wh{k(t)^a}$ are constructed. It is easy to see that if $K/k$ is an extension of analytic fields and $x,t\in K$ satisfy $|x-t|\le r_k(t)$, then $k(x)\toisom k(t)$ as valued fields (in particular, if $x\in k$, then $t=x$) and hence one obtains an isomorphism of analytic $k$-fields $\wh{k(x)}\toisom\wh{k(t)}$ sending $x$ to $t$. This is proved in \cite[Lemma~6.3.2]{temst} when $|x-t|<r_k(t)$ but the same argument applies when the equality holds. If $K$ is algebraically closed, then this extends further to an isomorphism of subfields $\wh{k(x)^a}\toisom\wh{k(t)^a}$ of $K$. In particular, we obtain the following result:

\begin{lem}\label{psilem}
Assume that $k$ is an algebraically closed analytic field and $K=\wh{k(t)^a}$. Then for any $k$-endomorphism $\psi\:K\to K$ the element $x=\psi(t)$ satisfies $|x-t|\le r_k(t)=r_k(x)$. Conversely, for any $x$ with $|x-t|\le r_k(t)$ there exists a $K$-endomorphism $\psi$ taking $t$ to $x$, and such $\psi$ is unique up to an action of $\Gal(\wh{k(t)}$.
\end{lem}

\subsubsection{Recollections on completed differentials}
Next, we recall the following relation between differentials and t.a. dependence.

\begin{lem}\label{diflem}
Let $K/k$ be an extension of analytic field and $t\in K$ such that $K=\wh{k(t)}$. Then the following conditions are equivalent:

(i) $k^s\cap K$ is dense in $K$ (in particular, $t\in\whka$),

(i) $\hatd_{K/k}(t)=0$,

(iii) $d_{\Kcirc/\kcirc}(t)$ is infinitely divisible.
\end{lem}
\begin{proof}
The equivalence of (i) and (ii) is covered by \cite[Lemma~2.4]{untilt}, and the equivalence of (ii) and (iii) follows from \cite[Lemma~4.1.11]{topdeg}.
\end{proof}

\subsubsection{Deeply ramified points of geometric type 4}
As in \cite{topdeg}, fancy extensions are constructed in this paper using deeply ramified extensions which force some differentials to vanish and hence lead to a t.a. dependence. More specifically, let $k$ be a non-trivially valued analytic field. It is proved in \cite[Theorem~5.1.12]{topdeg} that if $k$ is not perfectoid, then there exist points $z$ of type 4 with deeply ramified extension $\calH(z)/k$. Moreover, we are going to show that the proof in loc.cit. also yields the following refinement, which provides a control on the geometric radius of $z$ in a $k$-split disc.

\begin{lem}\label{deepramlem}
Let $k$ be an analytic field such that $p=\cha(\tilk)>0$ and the group $|k^\times|$ is dense. Assume that $t\in\kcirc$ is an element such that $d_{\kcirc}(t)$ is not infinitely divisible in $\Omega_{\kcirc}$ and $r\in(0,1)$. Then there exists a primitive extension $K=\wh{k(x)}$ of type 4 and an algebraic subextension $l/k$ of $K/k$ such that $d_{\lcirc}(t)$ is infinitely divisible, $|x|=1$ and $r_{\whka}(x)=r$.
\end{lem}
\begin{proof}
Fix $\pi\in k$ such that $|p|<|\pi|<1$, and if the characteristic is mixed, also replace $t$ by $p^{-n}t$ achieving that $|p|<|t|<1$. Choose a strictly increasing sequence of natural numbers $0<d_0<d_1<\dots$ such that $\prod_{n=o}^\infty |\pi|^{1/(p^{d_n}-1)}>r$. Now, we will construct an extension $l/k$ precisely as in the proof of \cite[5.1.12]{topdeg}: set $t_0=t$, take $t_{n+1}$ to be a root of $f_n(x)=x^{p^{d_n}}-\pi x-t_n$ and define $k_0=k$, $k_{n+1}=k_n(t_{n+1})$ and $l=\cup_n k_n$. A simple computation in the proof \cite[5.1.12]{topdeg} shows that $d_{\lcirc}t$, is infinitely divisible in $\Omega_{\lcirc}$ and  $r_\spl(t_{n+1}/k_{n})=|\pi|^{1/(p^{d_n}-1)}$.

Since $\prod_{n=0}^\infty r_\spl(t_{n+1}/k_{n})>r$ and $|k^\times|$ is dense, there exists a sequence $r_0=1,r_1,r_2,\dots$ of elements of $|k^\times|$ which converges to $r$ and satisfies $r_{n+1}/r_n<r_\spl(t_{n+1}/k_{n})$ for any $n\ge 0$. It then follows from \cite[Lemma~5.1.7]{topdeg} that starting with the unit disc $E_0=\calM(k\{x\})$ one can construct a sequence of discs $E_0\supset E_1\supset E_2\supset\dots$ such that $E_n$ is $k_n$-split and is of geometric radius $r_n$ in $E_0$. The intersection of these discs does not contain $k^a$-points and hence $\cap_n E_n=\{z\}$ is a single point of type 4 of geometric radius $r$. Thus, $K=\calH(z)=\wh{k(x)}$ contains $l$ and $r_{\whka}(x)=\inf_{c\in k^a}|c-x|_z=r$, as required.
\end{proof}

\subsubsection{Completed tame extensions}
In order to deal with discretely valued fields we will have to pass to an appropriate completed tame extension. This makes the group of values dense, and the following result shows that the divisibility properties of the differentials are not changed much.

\begin{lem}\label{tamelem}
Assume that $k$ is an analytic field, $\omega\in\Omega_{\kcirc}$ is an element, $l/k$ a tame algebraic extension and $K=\hatl$. Then $\omega$ is infinitely divisible if and only if the image of $\omega\otimes 1$ in $\Omega_{\Kcirc}$ is infinitely divisible.
\end{lem}
\begin{proof}
Since $K/k$ is separable, $H_1(\LL_{\Kcirc/\kcirc})=0$ by \cite[Theorem~5.2.3(ii)]{Temkintopforms} and hence the sequence
$$0\to\Omega_{\kcirc}\otimes_{\kcirc}\Kcirc\to\Omega_{\Kcirc}\to\Omega_{\Kcirc/\kcirc}\to 0$$ is exact. In the same way we have the exact sequence
$$0\to\Omega_{\lcirc/\kcirc}\otimes_{\kcirc}\lcirc\to\Omega_{\Kcirc/\kcirc}\to\Omega_{\Kcirc/\lcirc}\to 0,$$ whose first term is a torsion module annihilated by $\kcirccirc$, and whose third term is torsion free by \cite[Lemma~5.2.9]{Temkintopforms}. Thus, from the second sequence we obtain that $\Omega_{\Kcirc/\kcirc}$ contains no non-zero infinitely divisible torsion elements, and this combined with the first sequence implies the lemma.
\end{proof}

\subsubsection{Construction of towers with $\topdeg$ defect}
Now we can construct towers of simple extensions in which the topological transcendence degree is non-additive. This refines \cite[Theorems~5.2.2]{topdeg} by providing control on radii of generators, but we use essentially the same argument.

\begin{lem}\label{existlem}
Assume that $F/k$ is a non-trivial extension of algebraically closed analytic fields such that $F=\wh{k(t)^a}$ is non-trivially valued and $\cha(\tilk)>0$. Then for any positive $r<r_k(t)$ there exists an extension $K=\wh{F(x)^a}$ of $F$ such that $|x-t|<r_k(t)=r_k(x)$, $r=r_F(x)$ and $K=\wh{k(x)^a}$.
\end{lem}
\begin{proof}
Set $L=\wh{k(t)}$ and recall that $d_{\Lcirc/\kcirc}(t)$, and hence also $d_{\Lcirc}(t)$, is not infinitely divisible by Lemma~\ref{diflem}. Since $L$ can be discretely valued (when $k$ is trivially valued) we also consider the completed tame closure $L'=\wh{L^t}$, which has dense group of absolute values, and note that $d_{L'^\circ}(t)$ is not infinitely divisible by Lemma~\ref{tamelem}. Choose $\pi\in F$ such that $r<|\pi|<r_k(t)$. By Lemma~\ref{deepramlem} there exists an extension of analytic fields $E'=\wh{L'(y)}$ such that $d_{E'^\circ}(t)$ is infinitely divisible, $|y|=1$ and $r_F(y)=r/|\pi|$. Since $E'$ is a completed tame extension of $E=\wh{k(y,t)}$, we obtain by Lemma~\ref{tamelem} that already $d_{\Ecirc}(t)$ is infinitely divisible. Set $H=\wh{k(y)}$. Then $d_{\Ecirc/\Hcirc}(t)$ is infinitely divisible, and hence $t\in H$ by Lemma~\ref{diflem}.

We claim that $K=\wh{k(y)^a}$ and $x=t+\pi y$ are as required. First, $|x-t|=|\pi|<r_k(t)$, and hence $r_k(x)=r_k(t)$. Second, $r_F(x)=r_F(\pi y)=|\pi|r_F(y)=r$. Finally, $\wh{k(t,x)^a}=\wh{k(t,y})^a=K$. Since $\topdeg(K/k)=1$ and $K\neq\wh{k(t)^a}$, this implies that $t\in\wh{k(x)^a}$ and hence $K=\wh{k(x)^a}$.
\end{proof}

\begin{cor}\label{existcor}
Assume that $K/k$ is a non-trivial extension of algebraically closed analytic fields such that $K=\wh{k(t)^a}$ and $\cha(\tilk)>0$. Then for any positive $r<r_k(t)$ there exists a $k$-subfield $F=\wh{k(x)^a}\subsetneq K$ such that $r_F(t)=r$.
\end{cor}
\begin{proof}
By Lemma~\ref{existlem} there exists an extension $L=\wh{K(y)^a}$ such that $r_K(y)=r$, $|y-t|<r_k(t)$ and $L=\wh{k(y)^a}$. By Lemma~\ref{psilem} there exists a $k$-isomorphism $\psi\:L\toisom K$ taking $y$ to $t$. Then for $F=\psi(K)$ and $x=\psi(t)$ we have that $F=\wh{k(x)^a}\subsetneq K$ and $r_F(t)=r_K(x)=r$.
\end{proof}

\subsubsection{Applications}
Summarizing our results so far we obtain a surprisingly simple answer to \cite[Question~5.3.10]{topdeg}, namely the map $\rho_t$ from Lemma~\ref{subfieldcor} is a bijection.

\begin{theor}\label{Stheorem}
Let $K/k$ be an extension of algebraically closed analytic fields such that $\Topdeg(K/k)=1$, $\cha(\tilk)>0$ and $|K^\times|\neq 1$. Then any $t\in K$ with $\wh{k(t)^a}=K$ induces a bijection $\rho_t\:I_{K/k}\toisom[0,r_k(t)]$ sending an intermediate algebraically closed analytic field $F$ to $r_F(t)$. The bijection $\rho_t$ is monotonic and respects limits: $F\subseteq F'$ if and only if $r_F(t)\ge r_{F'}(t)$, and for any family $\{F_i\}\subset I_{K/k}$ with $F'=\cap_iF_i$ and $F''$ the completion of $\cup_i F_i$ one has that $r_{F'}(t)=\sup_i r_{F_i}(t)$ and $r_{F''}(t)=\inf_i r_{F_i}(t)$.
\end{theor}
\begin{proof}
The map $\rho_t$ is injective by Lemma \ref{subfieldcor} and surjective by Corollary~\ref{existcor}. Clearly, $\rho_t$ is monotonic and reverses the order, and hence it takes infima to suprema and vice versa. It remains to note that infima and suprema in $I_{K/k}$ are intersections and completed unions.
\end{proof}

\subsection{The PL structure of $I_{K/k}$}\label{Isec}
Throughout \S\ref{Isec} we assume that $k$ is algebraically closed, $p=\cha(\tilk)>0$, $k\subsetneq K=\wh{k(t)^a}$, $r=r_k(t)>0$, and we will study the natural structure the interval $I_{K/k}$ acquires.

\subsubsection{Log radius parameterizations}
Instead of the parametrization $\rho_t\:I_{K/k}\toisom[0,r]$ by the radius function, it will be convenient to use its logarithm, which will be denoted by $\rho^\rmlog_t\:I_{K/k}\toisom[-\log(r),\infty]$, where $\rho^\rmlog_t(F)=-\log(r_F(t))$. In particular, this parametrization is order preserving. Such functions also make sense for $t$ which is not a t.a. generator: if $F=\wh{k(t)^a}$, then $\rho^\rmlog_t\:[k,F]\toisom[-\log(r),\infty]$ and $\rho^\rmlog_t([F,K])=\infty$.

\subsubsection{Transition functions}
One can wonder to which extent the metric structure $\rho^\rmlog_t$ induces on $I_{K/k}$ is canonical. Since the parametrization is given by the log radius function $\rho^\rmlog_t$ of the t.a. generator $t$, an equivalent question is to describe the {\em transition functions}, that is, functions $[-\log r,\infty]\toisom[-\log r',\infty]$ of the form $\rho^\rmlog_{t,t'}=\rho^\rmlog_{t'}\circ(\rho^\rmlog_t)^{-1}$, where $t'$ is another t.a. generator and $r'=-\log r_k(t')$.

A tightly related question is as follows: what is the image of the restriction $$\Lam_t\:\Aut_k(K)\to\Aut(I_{K/k})\toisom\Aut([-\log r,\infty])?$$ Of course, this restriction takes an automorphism $\psi\in\Aut_k(K)$ to $\rholog_{t,\psi(t)}$.

\subsubsection{Immediate simple extensions}
The main tool of our study of the radii and transition functions will be the following results about simple immediate extensions $L=\wh{k(x)}$. They are not well recorded in the literature (though maybe known to experts), so we provide all needed details. Assume that $k$ is algebraically closed and $L/k$ is immediate. In particular, $r=r_k(x)=\inf_{a\in k}|x-a|$ is not attained. For shortness, set $x_a=x-a$ and $r_a=|x_a|$ for $a\in k$. For any element $f\in k[x]$ of degree $d$ let $f=\sum_{i=0}^d c_{a,i}x_a^i$ be its presentation with respect to the coordinate $x_a$.

\begin{lem}\label{immediatelem}
Keep the above notation, then

(i) $L=\wh{k[x]}$ and, moreover, $\Lcirc$ is the completion of $\cup_a\kcirc[x_a/\pi_a]$, where $\pi_a\in k$ is such that $|\pi_a|=r_a$.

(ii) There exists $r_0>r$ such that for any $a\in k$ with $r_a<r_0$ the numbers $s_i=|c_{a,i}|$ are independent of $a$ and satisfy the inequalities $s_j>\binom{i}{j}s_ir_a^{i-j}$ for any $0\le j<i\le d$ such that $s_j>0$.

(iii) One has that $r_k(f)=\max_{i>0}|s_i|r^i$ and $r_k(f)>|s_i|r^i$ whenever $i\notin \{0\}\cup p^\NN$. In particular, if $r_a$ is sufficiently close to $r$ and $i_1=p^N$ and $i_2=p^{N'}$ are the maximal and the minimal values of $i$ such that $r_k(f)=s_i|r^i|$, then $$f-c_{a,0}=g+\sum_{n=N'}^{N}c_{a,p^n}x_a^{p^n},$$ where $|g|<r_k(f)$.

(iv) If $F=\wh{k(f)}$, then $[L:F]=p^N$.
\end{lem}
\begin{proof}
(i) The valuation of $K$ corresponds to a point $z\in\AA^1_k$ on the Berkovich affine line and by Berkovich's classification of points on the affine line, $\{z\}=\cap_a D(a,r_a)$ is the intersection of discs given by $|x-a|\le r_a$. In particular, the valuation $|\ |_z$ on $k[x]$ is the infimum of the translated Gauss valuations $|\ |_{a,r_a}$ around $a$ of radius $r_a$, and the claim of (i) follows easily.

(ii) Let $\partial^j$ denote the $j$-th divided power derivative given by $\partial^j(f)=\sum_{i=j}^d\binom{i}{j}c_{j}x^{i-j}$ and choose $r_0$ so close to $r$ that the disc $D_0=D(a,r_0)$ of radius $r_0$ around $z$ does not contain zeros of any non-zero $\partial^j(f)$. Then each non-zero $\partial^j(f)$ is invertible in any disc contained in $D_0$ and hence its dominant term is the free one, that is $|c_{a,j}|>\binom{i}{j}|c_{a,i}|r_a^{i-j}$ for any $i>j$ (we use that $\partial^j(f)=\sum_{i=i}^d\binom{i}{j}c_{a,i}x_a^{i-j}$). Moreover, for any $b\in k$ we have that $c_{b,j}=\sum_{i=j}^d\binom{i}{j}c_{a,i}(b-a)^{i-j}$, hence whenever $|b-a|<r_0$ we have that $|c_{a,j}|=|c_{b,j}|$ is independent of $a$.

(iii) It follows from (i) that $r_k(f)$ is the infimum of the distance between $f$ and $k$ with respect to the Gauss valuations $|\ |_{a,r}$, that is, $$r_k(f)=\inf_{a\in k}\max_{i>0}|c_{a,i}|r_a^i.$$ Since $r=\inf_a r_a$ and the values of $|c_{a,i}|$ stabilise when $r_a$ is close enough to $r$, we obtain that $r_k(f)=\max_{i>0}|s_i|r^i$. If $i>0$ is not a $p$-th power, then $i=p^nm$ with $m>1$ and $(m,p)=1$. It follows that $j=p^n$ satisfies $p\nmid \binom{i}{j}$ and hence $s_ir^{i-j}<s_j$ by (ii). In particular, $s_ir^i<s_jr^i\le r_f(k)$. The last claim of (iii) follows.

(iv) Consider the morphism $h\:\AA^1_k\to\AA^1_k$ induced by $f$ and let $y=h(z)$. Then $\calH(z)=L$ and $\calH(y)=F$, so the claim reduces to showing that the local degree $\deg_h(z)$ of $h$ at $z$ is $p^N$. For $r'>r$ let $D_{r'}$ be the disc around $z$ of radius $r'$ and let $z_{r'}$ be its maximal point. By definition of $N$ in (iii), decreasing $r_0$ we can achieve that $s_ir'^i<s_{p^N}r'^{p^N}$ for any $r<r'<r_0$ and $0<i\neq p^N$. For any such $r'$, after translating the coordinates by $a$ and $c_{a,0}$ the morphism $D_{r'}\to h(D_{r'})$ becomes the morphism between discs with center at 0 given by a polynomial of order $p^N$ (recall that the order is the largest degree of a dominant term). By Weierstrass division theorem this morphism is of degree $p^N$, hence $\deg_h(z)\le p^N$. In addition, this morphism is precisely of degree $p^N$ at the maximal point of the disc. So, $\deg_h(z_r')=p^N$ and by semicontinuity of the degree we obtain the equality $\deg_h(z)=p^N$.
\end{proof}

It seems that a basic theory of almost tame extension is missing in the literature. Of course, this gap should be filled out, but perhaps in another venue. So, we will just state what we need without proof and will use in side remarks in the sequel.

\begin{rem}\label{remwild}
(i) A finite extension of valued fields $l/k$ is {\em almost tame} if $\Omega^\rmlog_{\lcirc/\kcirc}=0$. Otherwise we say that $l/k$ is {\em wildest}. Also, we say that $l/k$ is {\em purely wildest} if for any $k\subseteq k_1\subsetneq l_1\subseteq l$ the extension $l_1/k_1$ is wildest.

(ii) It is easy to see that the class of almost tame extensions is closed under compositions and taking subextensions. It includes the class of tame extensions and ``mildest'' wild extension and is, in fact, the first step of the ramification filtration beyond the tame class. A defectless extension $l/k$ is almost tame if and only if it is tame. A purely wild almost tame extension is immediate. Finally, $l/k$ contains a maximal almost tame subextension $l_{at}$. A bit subtler fact is that $l/l_{at}$ is purely wildest. We call $[l:l_{at}]$ the {\em wildest degree} of $l/k$. This is the analogue of the {\em wild degree} of $l/k$, which is the degree of $l$ over the maximal tame subfield $l_t$. Naturally, the wild and the wildest degrees are multiplicative in towers.

(iii) The meaningful interpretation of $N'$ in Lemma \ref{immediatelem} is that $p^{N'}$ is the wildest degree of $L/\wh{k(f)}$, while $p^N$ is its usual degree, which is also the wild degree.

(iv) If $L=\wh{k(t)}$ is not immediate over $k$, then $L$ is stable by \cite[Theorem~6.3.1(iii)]{temst} and hence any finite extension $F/L$ is defectless and its wild and wildest degrees coincide. In addition, the wild degree can be read off from the simple invariants -- it is the $p$-primary part of $e_{F/L}$ if $|L^\times|\neq|k^\times|$, and it is the inseparable degree of $\tilF/\tilL$ otherwise.
\end{rem}

\subsubsection{Explicit examples in positive characteristic}
Now we can construct a large family of explicit examples, which completely illustrate the general case. Technically, it is easier to deal with the case of positive characteristic, so assume in the next example that $\cha(k)=p$.

\begin{exam}\label{phiexam}
(i) For any $t'=\sum_{i=-\infty}^N c_it^{p^{i}}$ with $c_i\in k$, we have that $r_F(t')=\max_i(|c_i|r_F(t)^{p^i})$ for any $F\in I_{K/k}$. Indeed, to check this we can remove all terms of absolute value smaller than $r_F(t')$, and we can replace $t'$ by $t'^{p^m}$ because  $r_F(t'^{p^m})=(r_F(t'))^{p^m}$ (just apply $\Fr^m$). This reduces us to the case when $t'=\sum_{i=0}^Nc_it^{p^i}$, and then we can compute $r_F(t')$ by Lemma~\ref{immediatelem}(iii) because $c_i\in F$ and the coefficients $c_i$ do not change under any coordinate change $t'_a=t'-a$. Passing to the log scale we obtain that $$\rho^\rmlog_{t,t'}(s)=\min_i(p^is-\log|c_i|),$$ so the functions realized in this way as $\rho^\rmlog_{t,t'}$ are convex PL functions on $[-\log r,\infty]$ with slopes in $p^\ZZ$ and corners in $\log|k^\times|$ such that the sequence $-p^i\log|t|-\log|c_i|$ converges to infinity.

(ii) The same example with $t$ being the dominant term of $t'=t+\sum_{i=-\infty}^{-1} a_it^{p^{i}}$ can be used to construct explicit elements in the image of $\Lam_t$. Also, any PL function can be realized as a composition of a convex one and an inverse of a convex one, leading to a much reacher family of examples of the form $(\rholog_{t'',t'})^{-1}\circ\rho^\rmlog_{t,t'}=\rho^\rmlog_{t,t''}$.
\end{exam}

\subsubsection{Explicit examples in mixed characteristic}
Computations are more difficult in mixed characteristic, so we consider only the simplest case we will use later.

\begin{exam}\label{mixedphi}
(i) Let us study $\rholog_{t,t'}$, where $t'=t^p+ct$ with $c\in k$.  Using the presentations $t'-ca-a^p=ct_a+\sum_{i=1}^p\binom{p}{i}a^{p-i}t_a^i$ with $t_a=t-a$ and $a\in k$ such that $|a|=|t|$, one obtains that unless there is a cancellation $|c+pa^{p-1}|<|c|$, which does not happen when $|c|\neq|pa^{p-1}|$ and hence can be easily avoided in the sequel by translating $t$, the dominant term on the left is either $t_a^p$ or the linear term $(c+pa^{p-1})t_a$. So, in this case $\rholog_{t,t^p+ct}(s)=\min(ps,s+\gamma)$, where $\gamma=\min(-\log|c|,-\log|pt^{p-1}|))$. This function is linear if $s_1=\gamma/(p-1)$ is smaller than the endpoint $s_0=-\log r$ and otherwise has a single corner at $s_1$ and slopes $p$ and 1. The latter happens whenever $|c|<r^{p-1}$ and $|t|$ is sufficiently close to $r$ (see below). A technical complication in the mixed characteristic is that the corner cannot be too far from the endpoint: $$s_1-s_0\le-\frac{1}{p-1}\log|p|-\log|t|+\log r\le-\frac{1}{p-1}\log|p|$$ and the equality is obtained when $|c|<|pt^{p-1}|$, including the case of $c=0$ and $t'=t^p$. It follows easily that any PL function $\rho$ with a single corner $s_1\in|k^\times|$ such that $s_1-s_0<-\frac{1}{p-1}\log|p|$, $\rho(s)=ps$ for $s\in(s_0,s_1)$ and $\rho(s)$ is linear for $s>s_1$ can be realized as $\rho=\rholog_{t_a,t_a+ct_a^p}$ for appropriate $c\in k$ and $t_a=t-a$ with $|t_a|$ close enough to $r$.

(ii) Let now $\rho$ be any PL function on $[s_0,\infty)$ with two corners $s_1,s_2\in|k^\times|$ such that $s_1<s_2<s_0-\frac{1}{p-1}\log|p|$, and $\rho(s)=s$ for $s<s_1$, $\rho$ has slope $p^{-1}$ for $s\in(s_1,s_2)$ and is linear for $s>s_2$. It follows from (i) that choosing $c$ and $a$ appropriately we can find a realization $$\rho=\rholog_{t_a^{1/p},t_a+ct_a^{1/p}}\circ(\rholog_{t_a^{1/p},t_a})^{-1}=\rholog_{t_a,t_a+ct_a^{1/p}.}$$ Furthermore, when $|t_a|$ is close enough to $r$ the inequality $|c|<r^{\frac{p-1}p}$ strengthens to $|ct_a^{1/p}|<r$ and hence there exists an automorphism $\psi\in\Aut_k(K)$ taking $t_a$ to $t_a+ct_a^{1/p}$. By the construction $\rho=\Lam_t(\psi)$.
\end{exam}

\subsubsection{PL functions}
Let us formalize the properties of transition functions observed in the above examples. Assume that $\Gamma$ is a divisible subgroup of $\RR$. By a {\em $(p^\ZZ,\Gamma)$-PL function} on an interval $I\subseteq\RR$ we mean a continuous function $f\:I\to\RR$ which is locally of the form $f(t)=p^nt+\gamma$ with $n\in\ZZ$ and $\gamma\in\Gamma$ outside of a discrete set of {\em corner} points. In particular, $f$ is strictly increasing and each corner point $t_0$ lies in $\Gamma$ because it satisfies $p^nt_0+\gamma=p^{n_1}t_0+\gamma_1$ with $n\neq n_1$. By a {\em $(p^\ZZ,\Gamma)$-PL function} on an interval $[t_0,\infty]$ we mean an unbounded $(p^\ZZ,\Gamma)$-PL function $f$ on the ray $[t_0,\infty)$ augmented by the rule $f(\infty)=\infty$. In particular, the set of corner points is discrete in $[t_0,\infty)$ but can accumulate at infinity. Note that the set of $(p^\ZZ,\Gamma)$-PL functions is closed under compositions and taking inverses.

\subsubsection{The piecewise linearity theorem}
Now, let us establish the main result of \S\ref{Isec} that any transition function is, indeed, of this form.

\begin{theor}\label{PLth}
Let $K=\wh{k(t)^a}/k$ be an extension of algebraically closed analytic fields with $r=r_k(t)>0$, and assume that $|K^\times|\neq 1$ and $\cha(\tilk)=p>0$. Then:

(i) For any other t.a. generator $t'$ the function $\rho=\rho^\rmlog_{t,t'}$ is $(p^\ZZ,|K^\times|)$-PL.

(ii) If $|K^\times|\neq|k^\times|$, then $|K^\times|=|k^\times|\oplus r^\QQ$ and $\rho(-\log r)=-p^nb\log r-\log|a|$, where $a\in k^\times$, $b\in\ZZ\setminus p\ZZ$ and $p^n$ is the slope of $\rho$ at the endpoint $-\log r$.
%and the restriction map $\Aut_k(K)\to\Aut(I_{K/k})$ is the group of unbounded $(p^\ZZ,|K^\times|)$-PL functions which are identity in a neighborhood of $-\log r_t(k)$.
\end{theor}
\begin{proof}
We will use the observation that the set of PL functions satisfying conditions of the theorem is closed under compositions and taking the inverse function. Also, the claim of the theorem is local at a field $F\in I_{K/k}\setminus\{K\}$. So set $r_F:=r_F(t)$ and $s=-\log r_F$, and let us study $\rho$ locally at $s$. It will be convenient to work with a neighborhood $I'=[F_1,F_2]$ of $F$ in $I_{K/k}$ and shrink it when needed. In particular, we can assume that there exists $\veps>0$ such that $\rho_{t'}\ge \veps$ on $I'$, and hence $t'$ can be replaced by $t'-c$ with $c\in K$ and $|c|<\veps$ without affecting $\rho_{t'}$ and hence also $\rho$ on $I'$. In particular, we can assume that $t'\in k(t)^a$, and then $L=\wh{k(t,t')}$ is finite over $\wh{k(t)}$. Since $L$ is topologically finitely generated over $k$, by \cite[Theorem~6.3.1]{temst} there exists $t''\in L$ such that $L$ is unramified over $\wh{k(t''})$. Since $\rho$ is the composition of $\rho'=\rho_{t'',t'}^\rmlog$ and $\rho''=\rho^\rmlog_{t,t''}$ it suffices to prove the claim for $\rho'$ and $(\rho'')^{-1}$. Replacing $t$ by $t''$ for shortness of notation we can assume in the sequel that $L$ is unramified over $\wh{k(t)}$ itself. If $\tilF=\tilK$ this simply means that $L=\wh{k(t)}$. Finally, we will use that we can freely replace $t$ by $t_a=t-a$ with $a\in F$ without affecting $\rho$.

In the sequel we will have to separately consider three cases due to the type of the extension $L/k$. Recall that by Lemma~\ref{subfieldlem} if $K/F$ is not immediate, then $F=k$ is the endpoint of $I_{K/k}$.

{\it Type 4.} Assume first that $K/F$ is immediate. If $F\neq k$, then shrinking $I'$ we can assume that $k\notin I'$. In either case $K/F_1$ is immediate too, so $F_1[t]$ is dense in $L=\wh{F_1(t)}$ by Lemma~\ref{immediatelem}(i). Thus, moving $t'$ slightly we can even assume that $t'=\sum_ic_it^i\in F_1[t]$. By claim (iii) of the same lemma, there exists $r'>r_F$ such that for any $a\in F$ with $|t_a|<r'$ we have that $t'=c_{a,0}+g+\sum_{n=N}^{N'}c_{a,p^n}t_a^{p^n}$ with $|g|<r_F(t')$. Since $\rho_t$ and $\rho_{t'}$ are continuous, shrinking $I'$ we can assume that $|g|<\rho_{t'}$ and $\rho_{t}<r'$ on $I'$. In particular, there exists $a\in F_1\subseteq F$ such that $|t_a|<r'$ and for this $a$ we have that $c_{a,0}\in F_1$. Therefore this presentation of $t'$ and Lemma~\ref{immediatelem}(iii) can be used to compute $r_{E}(t')$ for any $E\in I'$, and we obtain that $\rho_{t'}(s)=\max_n|c_{a,p^n}|\rho_t^{p^n}(s)$ for any $s\in I'$. Thus, $\rho(s)=\min_n(p^ns-\log|c_{a,p^n}|)$ on $I'$, as required.

It remains to consider the case when $K/k$ is not immediate and automatically $F=k$. In this case, the $k$-norm on finite-dimensional $k$-subspaces of $K$ is Cartesian, hence translating $t$ and $t'$ by elements of $k$ we can achieve that $|t|=r=r_k(t)$ and $|t'|=r'=r_k(t')$. Now we again have to split to cases.

{\it Type 2.} Assume that $|k^\times|=|K^\times|$. Then we can also rescale $t$ and $t'$ by elements of $|k^\times|$ so that $r=|t|=|t'|=1$, while $\rholog_{t,t'}$ shifts by a value from $\log|k^\times|$. Thus we assume in the sequel that $\tilt,\tilt'$ are well defined and are not contained in $\tilk$. Let us start with two particular cases: (1) if $\tilt'=\tilt^{p^n}$, then $r'=|t'-t^{p^n}|<1$, hence $\rholog_{t,t'}=\rholog_{t,t^{p^n}}$ on $[0,-\log(r')]$, and the same computation as earlier shows that $\rholog_{t,t^{p^n}}$ has slope $p^n$ at $0$. (2) If $\tilL$ is separable over $\tilk(\tilt')$, then $L=\wh{k(t)}$ is unramified over $\wh{k(t')}$ because this extension is defectless by the stability theorem (e.g. see \cite[Theorem~6.3.1(iii)]{temst}). Therefore $\wh{F(t)}/\wh{F(t'})$ is unramified for any $F\in I_{K/k}$, and for $F\neq k$ this implies that $\wh{F(t)}=\wh{F(t')}$. As earlier, moving $t'$ a bit we can assume that $t'\in F[t]$, and then Lemma~\ref{immediatelem}(iv) implies that $t'=\sum_i a_it^i$ with $|a_1|r_F(t)>|a_i|r_F(t)^i$ for $i>1$. Thus, by the already established case of extensions of type 4, $\rholog_{t,t'}$ has slope 1 on $(-\log r,\infty)$ and hence also at the endpoint $-\log r$.

Now, assume that $t'$ is arbitrary. Since $\tilL$ is a one-dimensional function field over $\tilk$, the $p$-rank of $\tilL$ equals one and hence $\tilL$ is separable over a subfield of the form $\tilk(\tilt'^{1/p^n})$. Take any $t''\in L$ with $\tilt''=t'^{1/p^n}$, then $\rholog_{t'',t'}$ and $\rholog_{t,t''}$ are $(p^\ZZ,|k^\times|)$-PL at $s=0$ by the cases (1) and (2) above, and hence the same is true for $\rholog_{t,t'}$.

{\it Type 3.} Finally, assume that  $|k^\times|\neq|K^\times|$. Then $k\{t,t^{-1}\}_{r,r^{-1}}$ is already a field and hence coincides with $L$. Thus, $|L^\times|=|k^\times|\oplus r^\ZZ$ and moving $t'$ a bit we can assume that $t'=\sum_i c_it^i\in k[t^{\pm 1}]$. Then $|t'|=|c_mt^m|$ for a single $m$, and $m\neq 0$ because $|t'|=r_k(t')$. Thus $|t'-c_mt^m|<r'$ and we can replace $t'$ by $c_mt^m$ without affecting $\rho_{t'}$ on a small enough neighborhood $I'$ of $k$ in $I_{K/k}$. Furthermore, $\rho_{t'}=|c_m|\rho_{t^m}$, so it suffices to study $\rho_{t^m}$ on $I_{K/k}$. Write $m=bp^n$ with $(p,b)=1$. For any $E\in I_{K/k}$ the value of $r_E(t^m)$ is computed using the presentations $t^m=\sum_{i=0}^\infty\binom{m}{i}a^{m-i}t_a^i$ with $a\in E$, $|a|=|t|$ and $r_a=|t-a|$ tending to $r_E(t)$. Note that this formula covers the case of $m<0$ as well, and the minimal $i>0$ with $p\nmid\binom{m}{i}$ is $i=p^n$. Therefore, the $p^n$-th term is the dominant non-free one when $r_E(t)$ is close enough to $|t|=r_k(t)$, and in a small neighborhood of the endpoint of $I_{K/k}$ we have that $\rho_{t^m}(E)=|t|^{(b-1)p^n}\rho_t(E)^{p^n}$. That is, $\rho^\rmlog_{t,t^m}(s)=p^ns-(b-1)p^n\log(r)$ as (i) asserts and (ii) follows by substituting $s=-\log(r)$.
\end{proof}

\begin{rem}
Let $t,t'$ be as in the above theorem and assume that $t'\in k(t)^a$. We define the wild (resp. the wildest) degree of $t'$ over $t$ with respect to $F\in I_{K/k}$ as the ratio $n^w_{F,t/t'}$ (resp. $n^W_{F,t/t'}$) of the wild (resp. wildest) degrees of $\wh{F(t,t')}/\wh{F(t')}$ and $\wh{F(t,t')}/\wh{F(t)}$. Using Remark \ref{remwild} it is easy to see that the arguments above in fact establish the following more precise result: the slope of $\rholog_{t,t'}$ to the right (resp. left) of $s=r^\rmlog_t(F)$ equals $n^w_{F,t/t'}$ (resp. $n^W_{F,t/t'}$). In particular, even though at the endpoint $F=k$ the degrees $[L:\wh{k(t)}]$ and $[L:\wh{k(t')}]$ do not have to be a power of $p$, the slope only takes the wild degrees into account and hence lies in $p^\ZZ$.

In a sense, one can view $\wh{F(t)}$ as a ``lattice'' in $\wh{F(t)^a}$ and the transition functions measure the index between the lattices of $t$ and $t'$. If $\wh{F(t')}\subseteq\wh{F(t)}$, then the function is convex, and a corner point occurs at $F$ when this extension is not purely wildest. In general, we can only locally compare $\wh{F(t')}$ and $\wh{F(t)}$ by embedding them into a larger ``lattice'', so the transition function is the ratio of two convex $(p^\ZZ,|K^\times|)$-PL functions (that is, just a $(p^\ZZ,|K^\times|)$-PL function).
\end{rem}

\subsubsection{The PL structure}
Theorem \ref{PLth} implies that the ray $I_{K/k}\setminus\{K\}$ possesses a natural structure of a $(\ZZ[\frac 1p],|K^\times|)$-PL space. In fact, the structure is even more rigid since any change of coordinates has slopes in $p^\ZZ$. It is an interesting question if this non-typical restriction shows up elsewhere in the theory of PL spaces or their applications.

\subsubsection{Full characterization of transition functions}
Example \ref{phiexam} indicates that in the positive characteristic the PL structure we have just described is the finest natural structure on $I_{K/k}$. This holds in mixed characteristic too and can be formulated in the very precise form that the conditions of Theorem~\ref{PLth} are the only restrictions functions $\rholog_{t,t'}$ satisfy.

\begin{theor}\label{PLth2}
Let $K=\wh{k(t)^a}/k$ be as in Theorem~\ref{PLth}, then

(i) Any $(p^\ZZ,|K^\times|)$-PL function $\rho$ satisfying condition (ii) of Theorem~\ref{PLth} is of the form $\rholog_{t,t'}$ for an appropriate choice of $t'$.

(ii) The image of the map $$\Lam_t\:\Aut_k(K)\to\Aut(I_{K/k})\toisom\Aut([-\log r,\infty])$$ contains any $(p^\ZZ,\log|k^\times|)$-PL function $\rho$ on $[-\log r,\infty]$ which is the identity in a neighborhood of the endpoint $-\log r$.
\end{theor}
\begin{proof}
The idea is to compose PL functions with a single corner. First we observe that computations in the proof of Theorem~\ref{PLth} easily imply that for any $\rho$ as in (i) there exists $t''$ such that $\rho=\rholog_{t,t''}$ locally at the endpoint $s_0=-\log(r)$. Indeed, if $|k^\times|=|K^\times|$, then $\rho(s)=p^ns-\log|a|$ locally at $s_0$, and one can take $t'=at^{p^n}$. Otherwise locally at $s_0$ we have that $\rho(s)=p^ns-(b-1)p^n\log(r)-\log|a|$ with $b\in\ZZ\setminus p\ZZ$ and $a\in k^\times$, and one can take $t'=at^{bp^n}$. Now we can deduce (i) from (ii) as follows: $\rho\circ(\rholog_{t,t''})^{-1}$ is a $(p^\ZZ,\log|K^\times|)$-PL function, which is the identity in a neighborhood of $s_0$, hence if (ii) holds, then it comes from an automorphism $\psi$ and hence is of the form $\rholog_{t'',t'}$, where $t'=\psi(t'')$. Thus, $\rho=\rholog_{t'',t'}\circ\rholog_{t,t''}=\rholog_{t,t'}$.

The rest is devoted to the proof of (ii), so in the sequel we assume that $\rho$ is the identity in a neighborhood of $s_0$. First, suppose that the claim is proved in the particular case when $\rho$ has a single corner and let us deduce the general claim. By induction one easily finds functions $\rho_1, \rho_2,\dots$ such that $\rho_n$ has a single break at $s_n$, $\rho_n(s)=s$ for $s<s_n$ and $\rho_n\circ\dots\circ\rho_1$ coincides with $\rho$ on $[s_0,s_{n+1}]$. Choose any $F_i$ corresponding to a point $s'_i\in(s_{i-1},s_i)$, then by our assumption there exists $\psi_i\in\Aut_{F_i}(K)\subset\Aut_k(K)$ such that $\Lam_t(\psi_i)=\rho_i$. Therefore $\Lam_t(\psi_n\circ\dots\circ\psi_1)$ coincides with $\rho$ on $[s_0,s_n]$, and it remains to show that if the number of corners is infinite, then the infinite composition $\dots\psi_2\circ\psi_1$ converges. In this case $s'_i$ tend to infinity, hence $\rho_t(F_i)$ monotonically decrease to 0. In particular, the infinite composition converges to a (valuation preserving) $k$-automorphism of $\cup_i F_i$, and hence also of its completion $K$.

It remains to consider the case when $\rho$ has a single corner at a point $s_1\in\log|K^\times|$ with $s_1>s_0$. Let $p^n$ be the slope of $\rho$ for $s>s_1$. Then $\rho$ is the $-n$-th composition power of the function $\rho'$ which is the identity on $[s_0,s_1]$ and has slope $p^{-1}$ for $s>s_1$. Therefore, it suffices to show that $\rho'$ is in the image of $\Lam_t$ and for simplicity of notation we assume that $\rho$ itself has slope $p^{-1}$ for $s>s_1$.

We start with the simpler case when $\cha(k)=p$. Then by Example~\ref{phiexam} $\rho=\rholog_{t,t'}$, where $t'=t+ct^{1/p}$ and $c\in k$ is chosen so that $\log|c|=(p^{-1}-1)s_1$. Furthermore, $r>|cr^{1/p}|$ because $s_1>s_0$, hence translating $t$ by an element of $k$ and making $|t|$ close enough to $r$ we can also achieve that $r>|ct^{1/p}|$. Then there exists an automorphism $\psi$ taking $t$ to $t'$, and we have that $\Lam_t(\psi)=\rho$.

Finally, in the mixed characteristic case we will construct $\psi$ by composing functions $\psi_i\in\Aut_{F_i}(K)$ which take appropriate $t_a$ (with $a\in F_i$) to $t_a+ct_a^{1/p}$. Let $x_i=\rholog_t(F_i)\ge s_0$ denote the endpoint. The behaviour of a function $\rho_i=\Lam_t(\psi_i)$ was described in Example~\ref{mixedphi}: it has slopes $1,p^{-1},1$, and corners $x'_i,x''_i\in\log|K^\times|$ such that $x_i<x'_i<x''_i<x_i-\frac{1}{p-1}\log|p|$. We claim that $\psi_1,\psi_2,\dots$ can be chosen by induction so that each composition $\rho_n\circ\dots\circ\rho_1$ has slopes $1,p^{-1},1$, and corners at $s_1$ and $y_n$ so that $\lim_ny_n=\infty$. Indeed, one should make the choices so that the first corner is given by $x'_1=s_1$ and $x'_{n+1}=y_n$ for $n\ge 1$, the series $\sum_n(x''_n-x'_n)$ diverges (one can just take $x''_n=x'_n-\frac{1}{p}\log|p|$), and the endpoint $x_n$ satisfies $x''_n+\frac{1}{p-1}\log|p|<x_n<x'_n$. In this case $\lim_nx_n=\infty$, and hence the composition $\psi=\dots\circ\psi_2\circ\psi_1$ converges and satisfies $\Lam_t(\psi)=\dots\rho_2\circ\rho_1=\rho$.
\end{proof}

\subsection{A geometric interpretation}\label{geomsec}
One of initial motivations for this research was the question of Fargues-Fontaine whether any closed point of their (algebraic) curve has residue field isomorphic to $\CC_p$. Equivalently, whether all untilts of $\wh{\FF_p((t))^a}$ are isomorphic to $\CC_p$. This question was answered negatively in \cite{untilt} and our current results suggest a quantitative invariant which measures haw far it is from $\CC_p$. At this stage, this invariant does not seem to be too useful for studying the curve, nevertheless we prefer to describe the relevant geometry in the paper -- in the setting of the curve and in a few analogous but more standard cases. Throughout \S\ref{geomsec} we fix an algebraically closed analytic field $k$ and discuss Berkovich analytic picture, often over $k$, but the same examples work in adic geometry as well.

\subsubsection{Fields of definition}
Assume that $K$ is an analytic $k$-field and $X$ is a $K$-analytic space (so, $X$ is an analytic $k$-space in the terminology of \cite{berihes}). It often happens that the analytic space $X$ has other fields of definitions. In particular, in many cases one can deform $K$ inside of $\calO_X$ keeping $k$ fixed. This phenomenon is a direct analogy (in fact, a generalization) of the fact that fields of definition of complete local rings or non-reduced varieties are often non-unique. Here is a typical example, we will need later.

\begin{exam}\label{fiberexam}
(i) Consider two discs $X=\calM(k\{x\})$ and $Y=\calM(k\{y\})$, and let $Z=Y\times X=\calM(k\{x,y\})$ be their product. Fix the affine formal models $\gtZ=\gtX\times\gtY=\Spf(\kcirc\{x,y\})$, let $\gtz\in\gtZ_s$ be a generic point of the diagonal and let $Z_\gtz$ be the analytic fiber over $\gtz$, that is, the preimage of $\gtz$ under the reduction map $Z\to\gtZ_s$. Let $u,v$ be the maximal points of $X$ and $Y$, then the fiber $Z_u=Z\times_X u=\calM(\calH(u)\{y\})$ is a unit disc over $\calH(u)=\wh{k(x)}$, and similarly for $Z_v$, and $Z_\gtz$ is the open unit disc in both with coordinate $\pi=x-y$. Both $\calH(u)$ and $\calH(v)$ are natural fields of definition of the curve $Z_\gtz$, and there are many more choices, but no canonical one. Maybe the most illustrative way to present this example is $\kcirc\{x\}\llbracket \pi\rrbracket=\calO_Z^\circ(Z_\gtz)=\kcirc\{y\}\llbracket\pi\rrbracket$.

(ii) More generally, one can take $\gtz$ to be the generic point of a curve which maps dominantly on both $\gtX_s$ and $\gtY_s$, and then $Z_\gtz$ is given by $|P(x,y)|<1$, where $P\in k[x,y]$ is any polynomial whose reduction in $\tilk[\tilx,\tily]$ generates the ideal $m_\gtz$.
\end{exam}

\subsubsection{Comparison of geometric radii functions}
Let $Z=X\times Y=\calM(k\{x,y\})$, $\pi=x-y$ and $U=Z_\gtz=Z\{|\pi|<1\}$ be as in Example~\ref{fiberexam}(i). Set $p=\cha(\tilk)$, $K=\calH(u)=\wh{k(x)}$ and $L=\calH(v)=\wh{k(y)}$. The $K$-structure on the open unit disc $U$ with coordinate $\pi$ induces the geometric radius function $\olr_K=\olr_{\pi/K}\: U\to[0,1)$ given by $\olr_{K}(z)=\inf_{c\in K^a}|c-\pi|_z$. It is induced by the radius function on the $\whKa$-disc $\oU=U\wtimes_K\whKa$. In particular, the $K$-geometric type of a point $z\in U$ is 1 if and only if $\olr_K(z)=0$. In the same way, we have the geometric radius function $\olr_L$ induced by the $L$-structure. It is then a natural question to ask if these functions are compatible at least to a constant or a power, and if there is a natural (exponential) metric on $U$. We will later see that a canonical metric exists and coincides with $\olr_K$ and $\olr_L$ when $p=0$, but it is incomparable to the geometric ones when $p>0$.

At this moment let us only compare $\olr_K$ and $\olr_L$ when $p>0$ and the valuation on $k$ is non-trivial. Let $z\in U$ be a point and $F=\wh{\calH(z)^a}=\wh{k(x,y)^a}$. Note that $\olr_K(z)=0$ if and only if $\pi\in\whKa$ if and only if $F=\whKa=\wh{k(x)^a}$. Similarly, $\olr_L(z)=0$ if and only if $F=\whLa=\wh{k(y)^a}$. Thus it can freely happen that $\olr_K(z)=0\neq\olr_L(z)$, and the ratio $\olr_L/\olr_K$ blows up when we approach $z$. Furthermore, up to the action of $\Gal_K$, giving a point $z\in U$ of $K$-geometric type 1 is equivalent to fixing an element $y\in\whKa$ with $|y-x|<1$, and a $k$-endomorphism $\phi_z\:\whKa\to\whKa$ sending $x$ to $y$ is invertible if and only if $\olr_L(z)=0$. In general, $\olr_L(z)=\inf_{c\in\whLa}|x-c|_z$, which is the discrepancy between the metrics, measures how deeply $\whLa=\Im(\phi_z)$ sits inside $\whKa$. By Corollary~\ref{existcor} it can take any value in $[0,1)$, so $\olr_{x/L}$ induces a bit exotic function from the set of points of $\oU=U\wtimes_K\whKa$ of type 1 onto $[0,1)$. We will later explain that $\oU$ (and even $U\wtimes_K K'$ for any perfectoid $K$-subfield $K'\subseteq\whKa$) possesses unique metric and field of definition, so $\olr_{x/L}$ does not have a geometric meaning on $\oU$, but it comes from a meaningful (though not canonical) function on $U$.

\subsubsection{The curve}\label{curvesec}
Assume now that $k$ is trivially valued, for example, $k=\oFF_p$. In this case the above example also works once we replace $Z$ by a smaller (non-strict) polydisc $Z=\calM(\calA)$, where $\calA=k\{x,y\}_{r,r}=k\llbracket x\rrbracket\{y\}$ and $r<1$. In this case, $K=k((x))$, $L=k((y))$ and $\Gamma(\calO^\circ_\oU)=\oK^\circ\llbracket y\rrbracket$, where $\oK=\whKa$. In the same way $r_{y/L}$ defines a map from the set of points of $\oU$ of type 1 onto $[0,1)$, and the latter points can be identified with the closed points of $V=\Spec(B)\setminus V(x)$, where $B=\oK^\circ\llbracket y\rrbracket$.

As a mixed characteristics analogue one considers $V'=\Spec(B')\setminus V([x])$, where $B'=W(\oK^\circ)$, $p$ and the Teichmuller embedding $[\ ]\:\oK^\circ\into B'$ replace $B$, $y$ and the embedding $\oK^\circ\into B$, respectively. The set of closed points of $V'\setminus V(p)$ parameterizes untilts of $\CC_p$ and the function $r_{[x]/\CC_p}$, which maps this set onto $[0,1)$ via the rule $r_{[x]/\CC_p}(z)=\inf_{c\in\CC_p}|[x]-c|_z$, is a direct analogue of $\olr_{x/L}$. The only difference is that this time we do not have a natural descent as applying $W$ to non-perfect rings like $\Kcirc$ one does not get a reasonable (e.g. normal) ring of functions.

The curve $C$ of Hartl-Pink (introduced somewhat implicitly in \cite{Hartl-Pink}) is obtained by dividing (the analytic version of) $V\setminus V(y)$ by the Frobenius. The radius function is raised to the $p$-th power by the Frobenius, hence $\olr_{x/L}$ descends to $C$ only as a function $\olr_{x^{p^\ZZ}/L}$ with values in $[0,1)/\phi$, where $\phi$ is the $p$-th power map. In the same way, the Fargues-Fontaine $C'$ is obtained by dividing (the analytic version of) $V'\setminus V(p)$ by the Frobenius, $r_{[x]/\CC_p}$ descends to a function $r_{[x^{p^\ZZ}]/\CC_p}$ on the set of closed points of $C'$ with values in $[0,1)/\phi$, and an untilt $k(z)$ of $\oK$ coincides with $\CC_p$ if and only if $r_{[x^{p^\ZZ}]/\CC_p}(z)=0$.

\subsubsection{(Non)-uniqueness of fields of definition}
Below are two typical cases, when the field of definition is unique and non-unique, respectively.

\begin{lem}\label{uniquelem}
Let $K$ be an analytic $k$-field, $\calA$ a reduced strictly $K$-affinoid algebra provided with the spectral norm and $X=\calM(\calA)$.

(i) Assume that there exists an element $t\in K$ and a $k$-subfield $F\subset K$ such that $t\notin\wh{F^a}$ and $K$ is either a finite separable or a completed tame extension of $K_0=\wh{F(t)}$. Then there exists $\veps>0$ such that for any $t'\in\calA$ with $|t-t'|<\veps$ there exists a $k$-homomorphism $\phi\:K\into\calA$ such that $\phi(t)=t'$ and $|\phi(c)-c|<\veps|c|$ for any $c\in K$.

(ii) If $\calA$ is integral, $\cha(\tilk)>0$ and $K$ is perfectoid, then the algebraic closure of $K$ is $\calA$ is the unique maximal field of definition of $\calA$.
\end{lem}
\begin{proof}
(i) By assumption, $r=\inf_{c\in F^a}|t-c|>0$. If $K/K_0$ is a competed tame extension set $s=1$, and if $K/K_0$ is finite separable choose a primitive element $\alp$, let $r_{\alp/K_0}$ be the minimal distance between $\alp$ and its conjugates over $K_0$, and set $d_\alp=\inf_{c\in K_0}|c-\alp|$ and $s=r_\alp/d_\alp$. Now, let us show that $\veps=rs$ is as required.

First, extend the embedding $F\into\calA$ to $\phi\: F[t]\to\calA$ by sending $t$ to $t'$. By \cite[Lemma~3.1.6]{topdeg} for any $P\in F[t]$ and $x\in X$ one has that $|P-\phi(P)|_x=|P(t)-P(t')|_x\le s|P|$. Hence $|P-\phi(P)|<s|P|$ in $\calA$, and using the series to invert $1-\phi(P)/P$ one extends $\phi$ to $F(t)$, and hence also to $K_0$. Furthermore, it is easy to see that the extension also satisfies the inequality $|P-\phi(P)|<s|P|$. The final lift from $K_0$ to $K$ exists by Krasner's lemma (see \cite[Lemma~3.3.4]{insepunif}), where in case 2 we use that for any $\alp$ generating a tame extension of $K_0$ the ration $r_\alp/d_\alp$ equals one.

(ii) Replacing $K$ by a finite extension we can assume that it is algebraically closed in $\calA$, and our task is to prove that any other field of definition $K'$ algebraically closed in $\calA$ coincides with $K$. Since $\calA$ is strict, it is finite over a subalgebra $\calA'=K'\{t_1\.t_n\}$. Set $L=\Frac(\calA)$ and $L'=\Frac(\calA')$. It is easy to see that $\cap_n L'^{p^n}\subseteq K'$ and hence any element of $\cap_n L^{p^n}$ is algebraic over $K'$, yielding that $\cap_n L^{p^n}\subseteq K'$. Applying the same argument to $K=K^p$ we obtain that $\cap_n L^{p^n}=K\subseteq K'$, and therefore $K=K'$.

The mixed characteristic case can be established using the characterization of $K$ as the kernel of all bounded derivations on $\calA$. This uses that $\Omega_{\Kcirc}$ is divisible and hence $\hatOmega_K=0$, and we omit further details.
\end{proof}

We did not pursue the general case in the lemma, but a straightforward expectation is as follows.

\begin{rem}
(i) Let us assume that $K$ is non-trivially valued, $K\neq k$ and $\dim(X)>0$, ruling out the trivial obstacles to existence of a non-trivial $k$-deformation $\phi\:K\into\calA$ of the inclusion $K\subset\calA$. It seems plausible that in such a case no deformation exists if and only if $\cha(\tilk)>0$ and $K$ is perfectoid, which happens if and only if $\hatOmega_{K/k}=0$.

(ii) This is analogous to the following fact from commutative algebra: a complete local domain $A$ of characteristic $p$ and non-zero dimension possesses a unique field of definition $K\into A$ if and only if the residue field $K$ is perfect, and in this case the field of definition is $\cap_n A^{p^n}$. In the same spirit, if $K$ is perfect, then the set of Teichmuller elements of $R=W(K)$ coincides with $\cap_n R^{p^n}$.
\end{rem}

\subsubsection{Metrics on curves over algebraically closed fields}
Assume that $k$ is algebraically closed. Any $k$-curve $C$ possesses a canonical metric as follows. If $I\subset C$ is an interval with endpoints of type 2 or 3, the semistable reduction theorem implies that one can subdivide $I$ into intervals $I_1\.I_n$ which are skeletons of annuli $A_1\.A_n$ and we define the length by $\mu(I)=\sum_i M(A_i)$, where the modulus $M(A)$ of an annulus $A$ is defined by embedding it as $A=A(0;r_1,r_2)\subset\bbA^1_k$ and setting $M(A)=\log(r_2)-\log(r_1)$. This definition extends to arbitrary intervals by continuity, with $\mu(I)=\infty$ if and only if at least one endpoint is of type 1. Sometimes it is also convenient to use the exponential metric, so by the {\em radii} of $I$ and $A$ we mean $r(I)=e^{-\mu(I)}\in[0,1]$ and $r(A)=e^{-M(A)}=r_1/r_2$.

\begin{rem}\label{distrem}
It is easy to see that the above definition is independent of choices, but the most conceptual way is to observe that the metric is induced by restricting onto $I$ the sheaf of absolute values $|\calO_C^\times|=\calO_C^\times/(\calO^\circ_C)^\times$, whose stalk at $x$ is $|\calH(x)^\times|/|k^\times|$, and taking the logarithms. More concretely, $\mu(I)$ is the minimal length of an interval of the form $\phi(I)$, where $\phi\:I\to\bbR$ is strictly increasing and locally of the form $\log|f|$ with $f\in\calO_C^\times$ an invertible analytic function.
\end{rem}

\subsubsection{Geometric and canonical metrics}
If $k$ is arbitrary, then there are two natural metrics on a $k$-analytic curve $C$. First, the metric on $\oC=C\wtimes_k\whka$ is compatible with any automorphisms, hence descends to the {\em geometric metric} on $C$ -- for any interval $I=[a,b]\subset C$ choose a lift $\oI=[\oa,\ob]\subset\oC$ and define the geometric length and radius of $I$ by $\omu(I)=\mu(\oI)$ and $\olr(I)=r(\oI)$. This definition depends on the field of definition, and, as we have already seen, depends essentially.

On the other hand, the property in Remark \ref{distrem} is independent of $k$ and uses only logarithms of absolute values of functions. This leads to the definition of a canonical metric that we denote $\mu(I)$ and $r(I)$. By the definition, $\mu(I)\le\omu(I)$. It is not too difficult to show that $\mu=\omu$ if and only if either $\cha(\tilk)=0$ or $k$ is perfectoid, but this will be done elsewhere.

\begin{exam}
(i) Let $C=\bbA^1_k$ with coordinate $t$. Assume that $l=k(\alp)$ is of degree $p$ over $k$ and let $f_\alp\in k[t]$ be the minimal polynomial of $\alp$. Let $d_\alp=\inf_{c\in k}|c-\alp|$ and let $r_\alp$ be the distance between $\alp$ and any root of $f_\alp/(t-\alp)$. Let $R_\alp=[\alp,\infty)\subset C$ be the ray starting at $\alp$. The geometric radius parametrization of $R_\alp$ is given by $\olr(x)=\inf_{c\in k^a}|t-c|_x$, and we use it to identify $R_\alp$ with $\RR_{\ge 0}$. By Krasner's lemma $D(\alp,r_\alp)$ is a split $l$-disc, hence the geometric and the canonical metrics on $[0,r_\alp]$ coincide. The points of geometric radius $\olr>d_\alp$ are maximal point of $k$-split discs, so both metrics coincide on $[d_\alp,\infty)$ as well. Finally, the points in the interval $(r_\alp,d_\alp)$ are maximal points of non-split discs, and it is easy to see that these discs do not even contain Zariski closed points $x$ such that $k(x)/k$ is tame. It follows easily that any strictly increasing function of the form $|f(t)|$ on any subinterval $I$ is of the form $|P|/|Q|$, where $P,Q$ are polynomials of degree divisible by $p$. Therefore, $r(I)=\olr(I)^p$, $\mu(I)=p\cdot\omu(I)$ and the canonical metric on $R_\alp$ is obtained from the geometric metric by rescaling it by $p$ on $[r_\alp,d_\alp]$ and keeping unchanged for $\olr<r_\alp$ and $\olr>d_\alp$.

(ii) Similarly to the above, one can straightforwardly construct a nested sequence of discs $E_n$ of splitting degree $p^n$ so that the intersection point $z=\cap_n E_n$ is of geometric type 4 but the canonical distance to it is infinite. Thus $z$ is a singular point for the canonical metric but a regular point for the geometric one.

(iii) Moreover, in the situation of Example~\ref{fiberexam} with an algebraically closed $k$ one can use the same ideas as above to provide an elementary construction of a point $z\in Z_\gtz\subset Z=X\times Y=\bbA^2_k$ such that the projections $x\in X$ and $y\in Y$ are of types 1 and 4, respectively. This is equivalent to constructing a non-invertible automorphism of $\wh{k(x)^a}$. Our constructions of such automorphisms used computations with completed differentials, while the method we outline here is more straightforward and geometrically motivated.
\end{exam}

\section{Large extensions}\label{largesec}

\subsection{Large extensions of finite topological transcendence degree}
So far we have studied small extensions, which always have a finite topological transcendence basis $x_1\.d_d$ and satisfy $d=\topdeg(K/k)=\Topdeg(K/k)$. In this subsection we will study the remaining class of extensions of finite topological transcendence degree, the large ones. Recall that these are extensions satisfying $\topdeg(K/k)<\infty=\Topdeg(K/k)$, they do not possess a topological transcendence basis, and we already know from Corollary~\ref{equalcor} that they necessarily satisfy $\aleph_0<\Topdeg(K/k)$. Surprisingly, one can say quite a bit about such extensions and the mechanism preventing them from having a topological transcendence basis.

\subsubsection{Large extensions ot topological degree one}
The one-dimensional case is very concrete:

\begin{theor}\label{large1th}
Let $K/k$ be an extension of algebraically closed analytic fields such that $\topdeg(K/k)=1$ and $\cha(\tilk)>0$. Then the following conditions are equivalent:
\begin{itemize}
\item[(1)] The extension $K/k$ is large.
\item[(2)]  $\Topdeg(K/k)\ge\aleph_1$
\item[(3)]  $\Topdeg(K/k)=\aleph_1$
\item[(4)] The family $\{F_i\}_{i\in I}$ of intermediate algebraically closed analytic subextensions $k\subsetneq F_i\subsetneq K$ has uncountable cofinality.
\item[(5)] The family $\{F_i\}_{i\in I}$ has cofinality $\aleph_1$.
\item[(6)] $\omega_1$ is the only uncountable ordinal that can be embedded into $I$.
\item[(7)] $K$ is the (uncompleted) union of $F_i$.
\item[(8)] $K$ is a maximal extension of $k$ of topological degree one.
\item[(9)] Any $k$-endomorphism of $K$ is an automorphism.
\end{itemize}
\end{theor}
\begin{proof}
First, assume that the extension is not large, say, $K=\wh{k(x)^a}$, and let us show that all other conditions are not satisfied too. This is obvious for (2) and (3), conditions (4), (5), (6) fail by Theorem~\ref{subfieldth}, (7) fails by Lemma~\ref{cofinalem}, (8) fails by \cite[Theorem~5.2.2]{topdeg}, and (9) fails by \cite[Theorem~5.2.5]{topdeg}.

Conversely, assume that $K/k$ is large and let us deduce all other conditions. First, if (8) fails, then there exists an extension $K\subsetneq L$ such that $\topdeg(L/k)=1$ and then any $x\in L\setminus K$ is t.a. independent of $K$ and hence $K\subset\wh{k(x)^a}$, which contradicts Theorem~\ref{monotonic}. By maximality, any $k$-endomorphism of $K$ maps it onto itself, hence (9) holds too. By Theorem~\ref{unionth} (4) holds, and hence (7) holds by Lemma~\ref{cofinalem}. To deduce (6), and hence also (5), from (4) we should show that if $J$ is an uncountable well ordered subset of $I$, then $J$ is cofinal. Indeed, $\cup_{j\in J}F_j$ is a large extension of $k$ by the above paragraph, hence it is maximal by (8), and thus  $\cup_{j\in J}F_j=K$. Finally, by (6) there exists an uncountable cofinal family $\{F_j\}_{j\in J}$ of intermediate fields and $J$ is of cardinality $\aleph_1$. Since, each $F_j$ is not large, $F_j=\wh{k(x_j)^a}$, and hence the set $\{x_j\}_{j\in J}$ t.a. generates $K$ over $k$, and we obtain that $\Topdeg(K/k)\le\aleph_1$. Since $\Topdeg(K/k)$ is uncountable, (3) holds, and hence also (2).
\end{proof}

\begin{cor}
Assume that $K/k$ is a large extension of topological transcendence degree one. Then the ordered set $I_{K/k}$ of intermediate algebraically closed analytic fields is naturally homeomorphic to an Alexandroff ray (also known as a long ray) compactified by adding a point at infinity.
\end{cor}
\begin{proof}
By the above Theorem there exists a cofinal chain of fields $\{F_i\}_{i\in \omega_1}\subset I_{K/k}$. Hence $I_{K/k}\setminus\{K\}$ is the filtered union of $\omega_1$ closed intervals, which implies that $I_{K/k}\setminus\{K\}$ is a long ray and adding the maximal point $K$ compactifies it at infinity.
\end{proof}

\subsubsection{Cofinality}
It is less clear, what can be said about large extensions in general, but at least one can use the same argument to bound the cofinality of $I_{K/k}$.

\begin{theor}\label{cofinalityth}
Let $K/k$ be an extension of algebraically closed analytic fields such that $\topdeg(K/k)=d<\infty$ and $\cha(\tilk)>0$. Then the cofinality of the set $I_{K/k}$ does not exceed the ordinal $d\omega_1$.
\end{theor}
\begin{proof}
If the cofinality is larger, then we can find subfields $k=F_0,F_1\.F_d\in I_{K/k}$ such that $F_d\neq K$ and each interval $[F_i,F_{i+1}]=I_{F_{i+1}/F_i}$ in $I_{K/k}$ has uncountable cofinality. We claim that $\topdeg(K/F_{i+1})<\topdeg(K/F_i)$ for any $i$. Once this is proved, one obtains that $\topdeg(K/F_d)=0$ and hence $K=F_d$, giving a contradiction.

Without restriction of generality it suffices to show that $\topdeg(K/F)<d$ for $F=F_1$. Assume that this is not so, and hence there exists $x_1\.x_d\in K$ which are t.a. independent over $F$. For any $x_0\in F$ the elements $x_0\.x_d$ are t.a. dependent over $k$ and this leaves only one possibility: $x_0\in\wh{k(x_1\.x_d)^a}$. Thus, the small extension $L=\wh{k(x_1\.x_d)^a}$ of $k$ contains $F$ and hence $I_{L/k}$ has uncountable cofinality, which contradicts Theorem~\ref{subfieldth}.
\end{proof}

\begin{rem}
Assume that $K/k$ is a maximal extension of $k$ of topological transcendence degree $d$. Then the theorem easily implies that the cofinality of the set $I_{K/k}$ is $d'\omega_1$, where $1\le d'\le d$. Of course it can happen that $d'=d$. For example, this is the case when $K/k$ splits into a tower of $d$ large extensions of topological transcendence degree one. It is a natural question, if any positive $d'<d$ can occur as well.
\end{rem}

\subsection{Further questions}
I do not know about possible applications, but the natural further question to study is to which extent the same results hold for countably generated analytic extensions. In particular, this includes the following

\begin{question}
(i) Assume that $\Topdeg(K/k)=\aleph_0$. What is the maximal cofinality of a chain of intermediate algebraically closed analytic fields? Can it be uncountable?

(ii) Assume that $\topdeg(K/k)=\aleph_0$. What is the maximal cardinality of  $\Topdeg(K/k)$? Can it exceed $\aleph_1$?

(iii) Can it happen that $\aleph_1\le\topdeg(K/k)<\Topdeg(K/k)$?
\end{question}

\bibliographystyle{amsalpha}
\bibliography{pathologies}

\end{document}